\documentclass[11pt]{amsart}

\usepackage{amsfonts,amssymb,amsthm,amsmath,color,subfigure,tikz}
\usepackage{float,mathrsfs,graphicx}
\usepackage{epsfig}
\usepackage{caption}
\usepackage[all,cmtip]{xy}

\numberwithin{equation}{section}
\def\eps{\epsilon}
\def\R{{\bf R}}
\def\n{{\bf n}}

\def\ds{\displaystyle}
\def\bnu{\n}

\begin{document}
\title{Fick's law selects the Neumann boundary condition}
\author{Danielle Hilhorst}
\address[Danielle Hilhorst]{CNRS and Laboratoire de Mathématiques, University Paris-Saclay, 91405 Orsay Cedex, France}
\email{danielle.hilhorst@universite-paris-saclay.fr}

\author{Seung-Min Kang}
\address[Seung-Min Kang]{Department of Mathematical Sciences, KAIST,  291 Daehak-ro, Yuseong-gu, Daejeon, 34141, Republic of Korea }
\email{sngmn20@gmail.com}

\author{Ho-Youn Kim}
\address[Ho-Youn Kim]{Computer, Electrical and Mathematical Sciences and Engineering Division, 4700 King Abdullah University of Science and Technology, Thuwal, 23955-6900, Kingdom of Saudi Arabia}
\email{ghsl0615@gmail.com, hoyoun.kim@kaust.edu.sa}

\author{Yong-Jung Kim}
\address[Yong-Jung Kim]{Department of Mathematical Sciences, KAIST,  291 Daehak-ro, Yuseong-gu, Daejeon, 34141, Republic of Korea }
\email{yongkim@kaist.edu}

%

\theoremstyle{plain}
\newtheorem{theorem}{Theorem}
\newtheorem{proposition}[theorem]{Proposition}
\newtheorem{assumption}[theorem]{Assumption}
\newtheorem{lemma}[theorem]{Lemma}
\newtheorem{corollary}[theorem]{Corollary}
\newtheorem{definition}[theorem]{Definition}
\newtheorem{claim}[theorem]{Claim}
\newtheorem{remark}[theorem]{Remark}
\newtheorem{example}[theorem]{Example}

\newtheorem{problem}{\textcolor{red}{Problem}}
\newtheorem{question}[problem]{\textcolor{red}{Question}}

\numberwithin{equation}{section}
\numberwithin{theorem}{section}
\numberwithin{problem}{section}
\setcounter{tocdepth}{3}

\date{March 3, 2023}

\maketitle
\begin{abstract}
\medskip\noindent
We study the appearance of a boundary condition along an interface between two regions, one with constant diffusivity $1$ and the other with diffusivity $\eps>0$, when $\eps\to0$. In particular, we take Fick's diffusion law  in a context of reaction-diffusion equation with bistable nonlinearity and show that the limit of the reaction-diffusion equation satisfies the homogeneous Neumann boundary condition along the interface. This problem is developed as an application of heterogeneous diffusion laws to study the geometry effect of domain. 
\\

{\bf keywords}: Heterogeneous diffusion equation, Reaction-diffusion equation, Fick's law diffusion, Singular limit, Neumann boundary condition
\end{abstract}

\section{Introduction}

\subsection{Problem setup and results}

We consider an initial value problem for a reaction-diffusion equation,
\begin{equation} \label{Pe0}
\begin{cases}
u_t = \nabla\cdot(D_\eps \nabla u ) + f(u), &(t,x)\in (0,T)\times\Omega, \\
\n\cdot\nabla u =0, & (t,x)\in(0,T)\times\partial\Omega,\\
u(0,x) = u_0(x),&x\in\Omega,
\end{cases}
\tag{$P^\eps_0$}
\end{equation}
where $\Omega\subset\R^n$ is a bounded open set with a smooth boundary $\partial\Omega$, $\n$ is the outward unit normal vector on the boundary, and $u_0\in C(\overline{\Omega})$ is an initial value bounded by
\begin{equation}\label{1.1}
0\le u_0(x)\le M,\qquad  x\in\Omega,\quad M\ge1.
\end{equation}
The reaction function $f$ is continuously differentiable and bistable. More specifically, we take the following hypotheses; for $\alpha\in(0,1)$,
\begin{equation}\label{1.2}
\begin{cases}
f(0) = f(1) = f(\alpha) = 0, \ \  f'(0) <0, \, f'(1) < 0, \, f'(\alpha)>0,\\
f(u) < 0 \ \text{ for } u\in (0,\alpha)\cup(1,\infty), \quad f(u) > 0 \text{ for } u\in (\alpha,1).
\end{cases}
\end{equation}
A typical example for such a bistable non-linearity is $f(u) = u(u-\alpha)(1-u)$. Note that $u=0$ and $u=1$ are stable steady-states, $u=\alpha$ is unstable, and the upper bound $M$ in \eqref{1.1} is taken larger than the bigger stable steady-state $u=1$. The diffusivity $D_\eps$ is given as below. First, the domain $\Omega$ is divided into two parts as in Figure \ref{fig1}. The inner domain $\Omega_1$ ($\overline{\Omega}_1\subset\Omega$) is a connected open set with a smooth boundary denoted by $\Gamma:=\partial\Omega_1$. The outer domain $\Omega_0=\Omega\setminus\overline{\Omega}_1$ is also open and connected. The diffusivity $D_\eps$ is given by 
\begin{equation}\label{1.3}
D_\eps(x)=\begin{cases}
1, & x\in\Omega_1\cup\Gamma,\\
\eps,& x\in\Omega_0.
\end{cases}
\end{equation}

\begin{figure}
\centering
\includegraphics[width=0.6\textwidth]{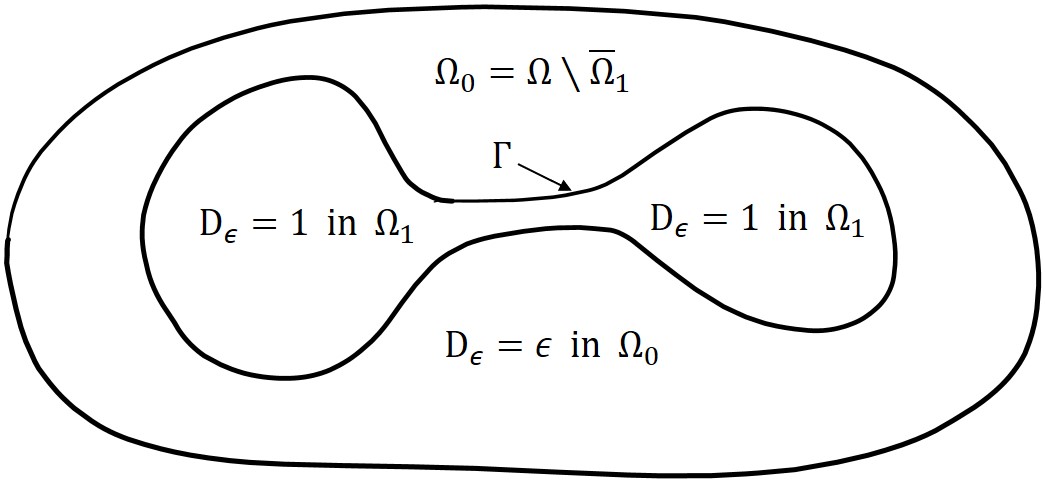}
\caption{\small Components of domain $\Omega$ and diffusivity $D_\eps$.}\label{fig1}
\end{figure}

Let $u$ be the limit of the solution $u^\eps$ of Problem \eqref{Pe0} as $\eps\to0$. Then, since the diffusivity in the inner domain $\Omega_1$ is fixed at $1$, the limit will obviously satisfy 
\[
u_t=\Delta u+f(u), (t,x)\in(0,T)\times\Omega_1.
\]
The missing part is the boundary condition. The purpose of the paper is to show that the limit $u$ is the unique solution of
\begin{equation}\label{P00}
\begin{cases}
u_t=\Delta u+f(u),& (t,x)\in(0,T)\times\Omega_1,\\
\bnu\cdot\nabla u = 0,& (t,x)\in (0,T)\times \partial\Omega_1,\\
u(0,x)=u_0(x),&x\in\Omega_1,
\end{cases}\tag{$P_0^0$}
\end{equation}
where $\bnu$ is the outward unit normal vector on $\Gamma(=\partial\Omega_1)$. Usually, since a boundary condition is needed to complete a second order problem, we often choose the Neumann or Dirichlet boundary conditions. In the paper, we show that the Neumann boundary condition appears naturally in the context of the Problem \eqref{Pe0}. Note that the solution behavior in one of the two domains is influenced by the one in the other domain due to the diffusion. This interaction eventually disappears if $\eps\to0$. However, traces of the interaction remain in the boundary condition at the interface $\Gamma$. The resulting boundary condition may depend on the types of diffusion and reaction function. The diffusion law in Problem \eqref{Pe0} is Fick's diffusion law, and the main conclusion of the paper is that Fick's law selects the Neumann boundary condition. We also prove that, in the outer domain $\Omega_0$, $u^\eps$ converges to the solution $u\in C^1([0,T]; L^\infty(\Omega_0))$ of the initial value problem
\begin{equation}\label{Q00}
	\begin{cases}u_t = f(u), & (t,x)\in(0,T)\times \Omega_0\\
		u(0,x) = u_0(x),&x\in\Omega_0,
	\end{cases}\tag{$Q_0^{0}$}
\end{equation}
which is an ODE system and a boundary condition is not needed.

The main results of the paper are in the following three theorems:
\begin{theorem}[Existence] \label{thm:existence} Let $\Omega\subset\R^n$ be a bounded domain with a smooth boundary, $u_0\in C(\overline\Omega)$ an initial value bounded by \eqref{1.1}, $f$ a continuously differentiable bistable non-linearity satisfying \eqref{1.2}, and $D_\eps$ the diffusivity given by \eqref{1.3}. Then, Problem \eqref{Pe0} possesses a unique weak solution $u^\eps$ and $\|\nabla u^\eps\|_{L^2((0,T)\times\Omega_1)}$ is uniformly bounded with respect to $\eps>0$.
\end{theorem}

\begin{theorem}[Neumann condition is selected] \label{thm:convergence1} Under the same assumptions as in Theorem \ref{thm:existence}, the solution $u^\eps$ of Problem \eqref{Pe0} converges to the unique solution of Problem \eqref{P00} strongly in $L^2((0,T)\times \Omega_1)$ as $\eps\to0$.
\end{theorem}
Moreover, we also prove the following convergence result.
\begin{theorem}[ODE solution]\label{thm:convergence2}
	Under the same assumption as in Theorem \ref{thm:existence}, the solution $u^\eps$ of Problem \eqref{Pe0} converges to the unique solution of Problem \eqref{Q00} weakly in $L^2((0,T)\times \Omega_0)$ as $\eps\to0$.
\end{theorem}

The organization of this paper is as follows. In Section 2, we introduce a sequence of perturbed problems $(P_\delta^\eps)$ where the diffusion coefficient $D_{\delta,\eps}$ is a smooth approximation of $D_\eps$. We present a priori estimates which are uniform in the parameters $\delta$ and $\eps$. In Section 3, we let $\delta$ tend to zero and deduce the solution existence of Problem \eqref{Pe0}. Then, the uniqueness of the problem is proved, which gives Theorem \ref{thm:existence}. Finally, we let $\eps$ tend to zero in Section 4 and present the proofs of Theorem \ref{thm:convergence1} and Theorem \ref{thm:convergence2}.

\subsection{Heterogeneous diffusion for geometry effects}

The paper is written under the motivation of a project to reinterpret the effect of domain geometry and boundary condition using heterogeneous diffusion. For example, let $u$ be a solution of a reaction-diffusion equation
\begin{equation}\label{E1}
\begin{cases}u_t = \Delta u + f(u), & (t,x)\in(0,T)\times \Omega_1,\\
\alpha u+(1-\alpha)\bnu\cdot\nabla u = 0,& (t,x)\in (0,T)\times \partial\Omega_1,\\
u(0,x) = u_0(x),&x\in\Omega_1,
\end{cases}
\end{equation}
where $f(u)$ is a bistable nonlinearity. The parameter $\alpha=0$ gives the Neumann boundary condition, and $\alpha=1$ the Dirichlet boundary condition. The solution behavior of the problem depends on the shape of domain and the boundary condition. Consider two examples. H. Matano showed in his seminal paper that, under the Neumann boundary condition ($\alpha=0$), nonconstant steady state solutions of the problem are unstable if $\Omega_1$ is convex \cite[Theorem 5.1]{001}. However, it can be stable if $\Omega_1$ is nonconvex \cite[Theorem 6.2]{001} (see \cite{002} and references therein for more works on dumbbell-shaped domains). On the other hand, Berestycki \emph{et al.} \cite{004} took the Neumann boundary condition and showed that a bistable traveling wave may propagate or be blocked depending on the exit shape of the domain, which is independent of the domain size.

To reinterpret the geometry effect, we may embed the domain $\Omega_1$ into the whole space $\Omega=\R^n$, and assign diffusivity $\eps>0$ to the outer domain $\Omega\setminus\Omega_1$. One might say that the solution of the whole domain problem converges to the solution of the original bounded domain problem as $\eps\to0$ and gives the geometric effect of the original problem. However, Theorem \ref{thm:convergence1} implies that Problem \eqref{Pe0} converges to a problem with the homogeneous Neumann boundary condition. A more general case involves the problem
\begin{equation}\label{E2}
\begin{cases}u_t = \nabla\cdot(D_\eps^{1-q}\nabla(D_\eps^q u)) + f(u), & (t,x)\in(0,T)\times \Omega,\\
\bnu\cdot\nabla u = 0,& (t,x)\in (0,T)\times \partial\Omega,\\
u(0,x) = u_0(x),&x\in\Omega.
\end{cases}
\end{equation}
If $q=0$, \eqref{E2} is identical to the Fick's law case in Problem \eqref{Pe0}. If $q=1$, the diffusion law is called Chapman \cite{Chapman}, and if $q=0.5$, Wereide \cite{Wereide}.  Diffusion laws with a general exponent $q$ may appear depending on the choice of reference points of the spatial heterogeneity (see \cite{Alfaro,012}). The resulting boundary condition depends on $q$. For example, the case without reaction ($f=0$) has been studied in one space dimension \cite{010}. The resulting boundary condition is Neumann if $q<0.5$ and Dirichlet if $q>0.5$. However, the critical exponent $q=q^*$ may depend on the reaction function $f$, the space dimension, and the regularity of the boundary, which requires further study.

\section{Classical solution for a smooth diffusivity} \label{sec:classical}

We start by regularising the diffusivity $D_\eps$ in \eqref{1.3} which is discontinuous along the interface $\Gamma$. For small parameters $\epsilon,\delta >0$, let
\begin{equation*}
D_{\delta,\eps}(x) = \begin{cases}
1, & x \in\Omega_1\cup\Gamma \\ 
\epsilon, & x \in\Omega_0\text{ and } \operatorname{dist}(\Gamma,x)>\delta,
\end{cases}
\end{equation*}
and then extend $D_{\delta,\eps}$ smoothly to the whole domain $\Omega$ so that $\eps\le D_{\delta,\eps}\le 1$. Then, it converges to the discontinuous diffusion coefficient pointwise, i.e.,
\[
\lim_{\delta \to 0} D_{\delta,\eps}(x) = D_{\eps}(x).
\]
We consider a regularized problem
\begin{equation} \label{Ped}
\begin{cases}
u_t=\nabla\cdot(D_{\delta,\eps}(x) \nabla u )+f(u), & (t,x) \in (0,T)\times\Omega\\
\n\cdot\nabla u =0, & (t,x)\in(0,T)\times\partial\Omega\\
u(0,x) = u_0(x),&x\in\Omega.
\end{cases}
\tag{$P_\delta^\eps$}
\end{equation}
Problem \eqref{Ped} formally converges to Problem \eqref{Pe0} as $\delta\to0$. We prove this below.

\begin{definition}
A function $u^{\delta,\eps} \in C^{1,2}((0,T) \times \Omega) \cap C([0,T]\times \overline{\Omega})$ is called a classical solution of Problem \eqref{Ped} if $u^{\delta,\eps}$ and its derivatives satisfy the partial differential equation in Problem \eqref{Ped} pointwise as well the initial and boundary conditions.
\end{definition}

\begin{lemma} \label{classicalsolution}
There exists a unique classical solution $u^{\delta,\eps}$ of Problem \eqref{Ped} which is bounded by
\begin{equation}\label{2.0}
0\le u^{\delta,\eps}\le M\quad\text{in } (0,T)\times \Omega.
\end{equation}
\end{lemma}
\begin{proof}
	The inequalities in \eqref{2.0} follow from the standard maximum principle and the initial condition \eqref{1.1}. Since $f(u) \le 0$ when $u = 0 , M$, $u=0$ is a lower solution and $u=M$ is a upper solution of \eqref{Ped}. 
	
	Now, the reaction function $f(u)$ is Lipschitz since it is a continuous function for bounded domain $0 \le u \le M$. We apply \cite[Theorem 4.2]{Pao} to deduce the existence of unique classical solution, and it completes the proof.
\end{proof}

Next, we obtain a priori estimates for the solution $u^{\delta,\eps}$.

\begin{lemma} \label{estimation1}
Let $\eps\in(0,1)$ and let $u^{\delta,\eps}$ be the classical solution of Problem \eqref{Ped} and define $M_f:= \sup_{s\in[0,M]}|f(s)|$. Then, we have the following inequalities.
\begin{eqnarray}
\label{2.1}
&&\|\nabla u^{\delta,\eps}\|_{L^2((0,T)\times\Omega_1)}\ \,\le C_1(T,|\Omega|,n,M,M_f),
\\
 \label{2.2}
&&\|\nabla u^{\delta,\eps}\|_{L^2((0,T)\times\Omega_0)}\, \le  \frac{1}{\sqrt{\eps}}C_2(T,|\Omega|,n,M,M_f) ,
\\ 
 \label{2.3}
&&\|u^{\delta,\eps}_t\|_{L^2(0,T;(H^{1})^*(\Omega))} \le C_3(T,|\Omega|,n,M,M_f).
\end{eqnarray}
\end{lemma}

\begin{proof}
We first prove \eqref{2.1} and \eqref{2.2}. We multiply the reaction-diffusion equation in Problem \eqref{Ped} by $u^{\delta,\eps}$ and integrate the result on $\Omega$. This gives
\[
\int_{\Omega} u^{\delta,\eps} u^{\delta,\eps}_t dx = \int_{\Omega} \nabla\cdot(D_{\delta,\eps} \nabla u^{\delta,\eps} ) u^{\delta,\eps} dx + \int_{\Omega} f(u^{\delta,\eps}) u^{\delta,\eps} dx.
\]
Integrating by parts gives
\[
{1\over 2} {d \over dt} \int_{ \Omega} (u^{\delta,\eps})^2 dx = - \int_{ \Omega} D_{\delta,\eps} |\nabla u^{\delta,\eps}|^2 dx  + \int_{ \Omega} f(u^{\delta,\eps}) u^{\delta,\eps} dx.
\]
Integrating the above equation on $[0,T]$ yields 
\begin{eqnarray*}
\ds{1\over 2} \int_{ \Omega} (u^{\delta,\eps}(T,x))^2 \ dx + \int_0^T \int_{ \Omega} D_{\delta,\eps} |\nabla u^{\delta,\eps}|^2   dx dt \\ 
\ds\hskip 20mm = {1\over 2} \int_{ \Omega} (u_0)^2  dx + \int_0^T \int_{ \Omega} f(u^{\delta,\eps}) u^{\delta,\eps}  dx dt.
\end{eqnarray*}
Since the first term on the left-hand-side is positive, we have
\begin{equation*}
\int_0^T \int_{ \Omega} D_{\delta,\eps} |\nabla u^{\delta,\eps}|^2   dx dt \le \frac{1}{2}\int_{ \Omega} (u_0)^2  dx + \int_0^T \int_{ \Omega} f(u^{\delta,\eps}) u^{\delta,\eps}  dx dt.
\end{equation*}
The bounds $0\le u^{\delta,\eps}\le M$ and $|f(u^{\delta,\eps})|\le M_f$ imply that
\begin{equation*}
\int_0^T \int_{ \Omega} D_{\delta,\eps} |\nabla u^{\delta,\eps}|^2  dx dt \le \frac{1}{2}|\Omega|M^2 + MM_f|\Omega|T.
\end{equation*}
Note that $D_{\delta,\eps} = 1$ in $\Omega_1$ and $\eps \le D_{\delta,\eps} \le 1$  in $ \Omega$, so that we have

\[
\int_0^T \int_{\Omega_1} |\nabla u^{\delta,\eps}|^2  dx dt \le \frac{1}{2}|\Omega|M^2 + MM_f|\Omega|T,
\]
and
\[
\int_0^T \int_{\Omega_0} |\nabla u^{\delta,\eps}|^2  dx dt \le \frac{1}{\eps}(\frac{1}{2}|\Omega|M^2 + MM_f|\Omega|T).
\]
Next we prove \eqref{2.3}. Let $\phi \in L^2(0,T;H^1( \Omega))$ be a test function. We multiply the partial differential equation in Problem \eqref{Ped} by $\phi$ and integrate on $(0,T)\times \Omega$ to obtain 
\[
\int_0^T \langle u^{\delta,\eps}_t , \phi\rangle dt= \int_0^T\int_{ \Omega} \nabla\cdot(D_{\delta,\eps} \nabla u^{\delta,\eps} ) \phi  dx dt+ \int_0^T \int_{ \Omega} f(u^{\delta,\eps}) \phi  dx dt,
\]
where $\langle\cdot,\cdot\rangle$ denotes the duality product between $H^1( \Omega)$ and $(H^1)^*( \Omega)$.
Applying integration by parts, we obtain
\[
\int_0^T \langle u^{\delta,\eps}_t,\phi\rangle dt = -\int_0^T \int_{ \Omega} D_{\delta,\eps}  \nabla u^{\delta,\eps} \cdot\nabla \phi dx dt + \int_0^T\int_{ \Omega} f(u^{\delta,\eps}) \phi  dx dt.
\]
The terms on the right hand side are bounded as follows.
\begin{align*}
	&\left|\int_0^T\int_{ \Omega} D_{\delta,\eps}  \nabla u^{\delta,\eps} \cdot\nabla \phi dx dt\right|\\
	\le& \left|\int_0^T\int_{ \Omega_1} \nabla u^{\delta,\eps}\cdot \nabla \phi dx dt\right|+\left|\int_0^T\int_{ \Omega_0} \eps \nabla u^{\delta,\eps}\cdot\nabla \phi dx dt\right|\\
	\le&  \Big(\int_0^T\int_{ \Omega_1} |\nabla u^{\delta,\eps}|^2 dx dt\Big)^{1\over 2}\Big(\int_0^T\int_{ \Omega_1} |\nabla \phi|^2 dx dt\Big)^{1\over 2} \\
	&+ \eps\Big(\int_0^T\int_{ \Omega_0} |\nabla u^{\delta,\eps}|^2 dx dt\Big)^{1\over 2}\Big(\int_0^T\int_{ \Omega_0} |\nabla \phi|^2 dx dt\Big)^{1\over 2}\\
	\le& (C_1+\sqrt{\eps}C_2) \Big(\int_0^T\int_{ \Omega} |\nabla \phi|^2 dx dt\Big)^{1\over 2},
\end{align*}
\[
\left|\int_0^T\int_{ \Omega} f(u^{\delta,\eps}) \phi  dx dt\right|\le M_f \sqrt{|\Omega|T} \Big(\int_0^T\int_{ \Omega} | \phi|^2   dx dt\Big)^{1\over 2}.
\]
We deduce that
\[
\Big|\int_0^T \langle u^{\delta,\eps}_t, \phi\rangle dt\Big| \le \Big(C_1+\sqrt{\eps}C_2 + M_f\sqrt{|\Omega|T}\Big)\|\phi\|_{L^2(0,T;H^1( \Omega))},
\]
which implies that
\begin{equation*}
\|  u^{\delta,\eps}_t\|_{L^2(0,T;(H^{1})^*( \Omega))} \le C_1+\sqrt{\eps}C_2 + M_f\sqrt{|\Omega|T}\le C_3(T,|\Omega|,n,M,M_f).
\end{equation*}
This completes the proof of \eqref{2.3}.
\end{proof}

\section{Existence and uniqueness of a weak solution of \eqref{Pe0}}

In this section, we prove that the solution of Problem \eqref{Pe0} exists by taking the singular limit as $\delta \to 0$ in Problem \eqref{Ped}. The solution of Problem \eqref{Pe0} is defined in a weak sense.

\begin{definition}\label{def:weak}
A function $u\in C([0,T];L^2( \Omega)) \cap L^2((0,T);H^1( \Omega)) \cap $\\ $L^\infty((0,T) \times  \Omega)$ is called a weak solution of Problem \eqref{Pe0} if
\begin{eqnarray}\label{eq:weak}
&\ds\int_0^T \int_{ \Omega} (-u \phi_t + D_\eps(x) \nabla u \cdot \nabla \phi - f(u) \phi) dx dt\nonumber\\
&\ds = \int_{ \Omega} u_0(x) \phi(0,x) - u(T,x)\phi(T,x)dx,
\end{eqnarray}
for all test functions $\phi\in H^1((0,T)\times \Omega)$.

\end{definition}

In order to show the existence of a weak solution of Problem \eqref{Pe0}, we consider the solution $u^{\delta,\eps}$ of Problem \eqref{Ped} and apply Fr\'echet-Kolmogorov theorem which is introduced in \cite[Proposition 2.5]{Crooks}, \cite[Theorem 4.26, p.111, Corollary 4.27, p.113]{Brezis}. In the following proposition, $Q_T:= (0,T)\times \Omega$ and $\bar{\eps}$ and $\bar{\delta}$ are small constants which are not related to $\eps$ and $\delta$ in Problem \eqref{Ped}.
\begin{proposition}[Fr\'echet-Kolmogorov] \label{frechet}
A bounded set $B\subset L^2(Q_T)$ is precompact in $L^2(Q_T)$ if
\begin{enumerate}
\item For any $\bar{\eps}>0$ and any subset $Q \Subset Q_T$, there exists a $\bar{\delta}>0$ such that $\bar{\delta}<\operatorname{dist}(Q,\partial Q_T)$ and
\[
\|u(t+\tau,x)-u(t,x)\|_{L^2(Q)} + \|u(t,x+\xi)-u(t,x)\|_{L^2(Q)} < \bar{\eps}
\]
for all $\tau, \xi,$ and $u \in B$ whenever $|\tau|+|\xi| < \bar{\delta}$.

\item For any $\bar{\eps}>0$, there exists $Q  \Subset Q_T$ such that $$\|u\|_{L^2(Q_T \backslash Q)} < \bar{\eps}$$ for all $u \in B$.
\end{enumerate}
\end{proposition}

Next, we apply the Fr\'echet-Kolmogorov theorem to the collection of classical solutions $\{u^{\delta,\eps}\}$ of \eqref{Ped} and show that it is precompact in $L^2(Q_T)$. We recall that the constant $\eps>0$ is fixed through this section. Since the classical solutions $\{u^{\delta,\eps}\}$ are uniformly bounded, the second assertion of Proposition \ref{frechet} can be easily derived. Indeed, let $Q = (0,T-\tau) \times \Omega^r$ for some small $\tau, r >0$, where $\Omega^r =\{x\in \Omega:\operatorname{dist}(x,\partial\Omega) > r \}$. Then,
\[
\|u^{\delta,\eps}\|_{L^2( Q_T \backslash Q)}^2 \le \int_{T-\tau}^T \int_{ \Omega} (u^{\delta,\eps})^2   dx dt + \int^{T}_0 \int_{ \Omega\backslash \Omega^r} (u^{\delta,\eps})^2   dx dt.
\]
From Lemma 2.3, we have 
\begin{equation}\label{ineq1}
\int_{T-\tau}^T \int_{ \Omega} (u^{\delta,\eps})^2   dx dt \le \tau|\Omega| M^2, 
\end{equation} and 
\begin{equation}\label{ineq2}
\int^{T}_0 \int_{ \Omega\backslash \Omega^r} (u^{\delta,\eps})^2   dx dt \le |\Omega\backslash\Omega_r|TM^2.
\end{equation}
Note that the right-hand sides of the inequalities \eqref{ineq1} and \eqref{ineq2} tend to zero as $\tau \to 0$ and $r \to 0$. Thus for any $\bar{\eps}>0$, we may choose $\tau>0$ and $r>0$ so small that 
\[
\|u^{\delta,\eps}\|_{L^2( Q_T \backslash Q)}^2 \le \tau |\Omega| M^2 + |\Omega\backslash \Omega^r|TM^2 < \bar{\eps}.
\]
This completes the proof of the property (2) in the Fr\'echet-Kolmogorov theorem.

We now estimate the equicontinuity of time and space variable for $\{u^{\delta,\eps}\}$ (see \cite[Lemmas 2.6 and 2.7]{Crooks}). 
\begin{lemma} For any small positive constant $r>0$,
\begin{enumerate}
\item There exists a positive constant $C_\eps$ which is independent of $\delta$ such that
\begin{equation}\label{3.4}
\int_0^{T} \int_{ \Omega^r} (u^{\delta,\eps}(t,x+\xi)-u^{\delta,\eps}(t,x))^2 dx dt \le C_\eps |\xi|^2  ,
\end{equation}
for all real values $|\xi| \le r$.
\item There exists a positive constant $C$ which is independent of $\delta$ and $\eps$ such that
\begin{equation} \label{3.5}
\int_0^{T} \int_{\Omega_1^r} (u^{\delta,\eps}(t,x+\xi)-u^{\delta,\eps}(t,x))^2 dx dt \le C |\xi|^2  ,
\end{equation}
for all real values $|\xi| \le r$.
\end{enumerate}
 
\end{lemma}
\begin{proof}
We first prove \eqref{3.4}. 
\begin{eqnarray*}
&&\int_0^{T}\int_{\Omega^r} (u^{\delta,\eps}(t,x+\xi)-u^{\delta,\eps}(t,x))^2 dx dt \\ 
&&\qquad= \int_0^{T} \int_{\Omega^r} \Big(\int_0^1 \nabla u^{\delta,\eps}(t,x+\theta \xi)\cdot \xi d\theta \Big)^2   dx dt \\ 
&&\qquad\le |\xi|^2\int_0^1 \int_0^{T} \int_{\Omega^r}  | \nabla u^{\delta,\eps}(t,x+\theta \xi)|^2   dx dt d\theta \\
&&\qquad\le |\xi|^2 \int_0^1 \int_0^T \int_{\Omega}  | \nabla u^{\delta,\eps}(t,x)|^2   dx dt d\theta =  |\xi|^2 \| \nabla u^{\delta,\eps} \|_{L^2(Q_T)}^2
\end{eqnarray*}
Thus we deduce from \eqref{2.1} that
\[
\int_0^{T} \int_{ \Omega} (u^{\delta,\eps}(t,x+\xi)-u^{\delta,\eps}(t,x))^2   dx dt \le (C_1^2+C_{2,\eps}^2) |\xi|^2,
\]
where $C_1$ is the upper bound in \eqref{2.1} and 
\begin{equation}\label{3.6}
C_{2,\eps}=\displaystyle{\frac{1}{\sqrt{\eps}}}C_2,
\end{equation}
 where $C_2$ is the upper bound in \eqref{2.2} which are  independent of $\delta$.

Next we prove the inequality \eqref{3.5}.
\begin{eqnarray*}
&& \int_0^{T} \int_{\Omega_1^r} (u^{\delta,\eps}(t,x+\xi)-u^{\delta,\eps}(t,x))^2 dx dt \\ 
&&\qquad= \int_0^{T} \int_{\Omega_1^r} \Big(\int_0^1 \nabla u^{\delta,\eps}(t,x+\theta \xi)\cdot \xi d\theta \Big)^2   dx dt \\ 
&&\qquad\le |\xi|^2\int_0^1 \int_0^{T} \int_{\Omega_1^r}  | \nabla u^{\delta,\eps}(t,x+\theta \xi)|^2   dx dt d\theta \\
&&\qquad\le |\xi|^2 \int_0^1 \int_0^T \int_{\Omega_1}  | \nabla u^{\delta,\eps}(t,x)|^2   dx dt d\theta =  |\xi|^2 \| \nabla u^{\delta,\eps} \|_{L^2((0,T)\times \Omega_1)}^2,
\end{eqnarray*}
which in view of the inequality \eqref{2.2} implies that
\[
\int_0^{T} \int_{\Omega_1^r} (u^{\delta,\eps}(t,x+\xi)-u^{\delta,\eps}(t,x))^2   dx dt \le C_1^2 |\xi|^2,
\]
where $C_1$ is independent of $\delta$ and $\eps$.
\end{proof}

\begin{lemma}
For any small positive constant $r>0$,
\begin{enumerate}
\item There exists a positive constant $C_\eps$ which is independent of $\delta$ such that
\[
\int_r^{T-r} \int_{ \Omega} (u^{\delta,\eps}(t+\tau,x)-u^{\delta,\eps}(t,x))^2   dx dt \le C_\eps \tau  ,
\]
for all real values $\tau$ with $0<\tau \le r < T$.
\item There exists a positive constant $C$ which is independent of $\delta$ and $\eps$ such that
\begin{equation} \label{time trans omega1}
\int_r^{T-r} \int_{\Omega_1^r} (u^{\delta,\eps}(t+\tau,x)-u^{\delta,\eps}(t,x))^2   dx dt \le C \tau  ,
\end{equation}
for all real values $\tau$ with $0<\tau \le r < T$.
\end{enumerate}

\end{lemma}
\begin{proof} 
We have
\begin{eqnarray*}
	&&\int_r^{T-r} \int_{\Omega} (u^{\delta,\eps}(t+\tau,x)-u^{\delta,\eps}(t,x))^2   dx dt \\ 
	&=& \int_r^{T-r} \int_{\Omega} (u^{\delta,\eps}(t+\tau,x)-u^{\delta,\eps}(t,x)) \Big(\int_0^\tau u^{\delta,\eps}_t(t+s,x) ds \Big)   dx dt.\\
	&\le& \left|\int_0^\tau \int_r^{T-r}\langle u_t^{\delta,\eps}(t+s,x),u^{\delta,\eps}(t+\tau,x)\rangle  dt ds\right|\\
	&&\quad+\left|\int_0^\tau \int_r^{T-r}\langle u_t^{\delta,\eps}(t+s,x),u^{\delta,\eps}(t,x)\rangle dt ds\right|\\
	&\le& 2\tau \|u_t^{\delta,\eps}\|_{L^2(0,T;(H^1)^*(\Omega))}\|u^{\delta,\eps}\|_{L^2(0,T;H^1(\Omega))}\\
	&\le& 2\tau C_3\sqrt{|\Omega|TM^2 + C_1^2+C^2_{2,\eps}},
\end{eqnarray*}
where $C_1$, $C_{2,\eps}$, $C_3$ are the upper bounds in \eqref{2.1}, \eqref{3.6}, \eqref{2.3}, respectively, which are independent of $\delta$.
%
%

Next we prove \eqref{time trans omega1}. Let $\mu(x)\in C_c^\infty(\Omega_1)$ be such that $0\le\mu\le1$ in $\Omega_1$ and $\mu=1$ on $\Omega_1^r$. We extend $\mu$ by $0$ on $\Omega_0$. Then we have
\begin{eqnarray*}
	&&\int_r^{T-r} \int_{\Omega_1^r} (u^{\delta,\eps}(t+\tau,x)-u^{\delta,\eps}(t,x))^2   dx dt \\ 
	&\le& \int_r^{T-r} \int_{\Omega}\mu(x)   (u^{\delta,\eps}(t+\tau,x)-u^{\delta,\eps}(t,x))^2   dx dt\\
	&=& \int_r^{T-r} \int_{\Omega}\mu(x)   (u^{\delta,\eps}(t+\tau,x)-u^{\delta,\eps}(t,x)) \Big(\int_0^\tau u^{\delta,\eps}_t(t+s,x) ds \Big) dx dt\\
	&=&\left|\int_0^\tau \int_r^{T-r} \langle u_t^{\delta,\eps}(t+s,x),\mu(x)u^{\delta,\eps}(t+\tau,x)\rangle  dt ds\right|\\
	&&+\left|\int_0^\tau \int_r^{T-r} \langle u_t^{\delta,\eps}(t+s,x),\mu(x)u^{\delta,\eps}(t,x)\rangle  dt ds\right|\\
	&\le&2\tau \|u_t^{\delta,\eps}\|_{L^2(0,T;(H^1)^*(\Omega))}\|\mu u^{\delta,\eps}\|_{L^2(0,T;H^1(\Omega))}.
\end{eqnarray*}
We remark that
\begin{align*}
\|\mu u^{\delta,\eps}\|_{L^2(0,T;H^1(\Omega))}=&
\big(\int_0^T \int_{\Omega} \mu^2(u^{\delta,\eps})^2 + \mu^2|\nabla u^{\delta,\eps}|^2 + (u^{\delta,\eps})^2|\nabla\mu|^2 dx dt\big)^{1/2}\\
\le& \big( \int_0^T \int_{\Omega_1} (u^{\delta,\eps})^2 + |\nabla u^{\delta,\eps}|^2 + M^2|\nabla\mu|^2 dxdt \big)^{1/2}\\
\le&\sqrt{TM^2(|\Omega_1|+\|\nabla\mu\|_{L^2(\Omega_1)}^2)+C_1^2},
\end{align*}
where $C_1$ is the upper bound in \eqref{2.1}. Then we obtain
\begin{eqnarray*}
 &&\ds\int_r^{T-r} \int_{\Omega_1^r} (u^{\delta,\eps}(t+\tau,x)-u^{\delta,\eps}(t,x))^2   dx dt \\
 &&\qquad\le 2\tau C_3\sqrt{TM^2(|\Omega_1|+\|\nabla\mu\|_{L^2(\Omega_1)}^2)+C_1^2},
\end{eqnarray*}
which is independent of $\delta$ and $\eps$.
\end{proof}

Thus we conclude that $\{u^{\delta,\eps}\}$ is precompact in $L^2( Q_T)$ and that there exists a function $u^\eps  \in L^2( Q_T)$ such that $u^{\delta,\eps}$ converges strongly in $L^2( Q_T)$ along a subsequence. We are ready to show that the function $u^\eps(t,x)$ is a weak solution of Problem \eqref{Pe0}.

\begin{proof}[\bf Proof of Theorem \ref{thm:existence}]

From the Fr\'echet-Kolmogorov theorem, we deduce that there exists a function $u^\eps  \in L^2( Q_T)$ and a subsequence $\{u^{\delta_i,\eps}\}$, which we denote again by $u^{\delta,\eps}$ such that
$$u^{\delta,\eps} \to u^\eps \quad \text{ strongly in } L^2((0,T) \times \Omega)\text{ and a.e in $Q_T$}\quad \text{as $\delta\to 0$}.$$ 

Moreover, we deduce from \eqref{2.1} that 
\[u^{\delta,\eps}\rightharpoonup u^\eps\quad\text{weakly in } L^2(0,T;H^1(\Omega)) \quad\text{as }\delta\to 0,\]
and it follows from \eqref{2.3} that
\[u^{\delta,\eps}_t\rightharpoonup u^\eps_t\quad\text{weakly in } L^2(0,T;(H^1)^*(\Omega))\quad\text{as }\delta\to 0,\]
so that $u^\eps\in C([0,T];L^2(\Omega))$. Next we show that
\begin{equation}\label{3.7}
u^{\delta,\eps}(t,\cdot)\rightharpoonup u^\eps(t,\cdot)\mbox{ weakly in $(H^1)^*(\Omega)$, for all $0\le t\le T$.}
\end{equation} 
Indeed, for any $0\le\tau\le T$, for all test function $\phi\in H^1(\Omega)$,
\begin{equation}\label{3.8}
	\int_\Omega (u^{\delta,\eps}(\tau,x)-u_0(x))\phi dx=\int_0^\tau \int_\Omega u^{\delta,\eps}_t \phi dx dt =\int_0^\tau \langle u^{\delta,\eps}_t,\phi\rangle  dt.
\end{equation}
We remark that 
\begin{equation}\label{3.9}
	\lim_{\delta\to0}\int_0^\tau \langle u^{\delta,\eps}_t,\phi\rangle  dt = \int_0^\tau \langle u^{\eps}_t,\phi\rangle  dt=\int_\Omega (u^{\eps}(\tau,x)-u_0(x))\phi dx .
\end{equation}
We deduce from \eqref{3.8} and \eqref{3.9} that
\begin{equation}\label{3.10}
	\lim_{\delta\to0} \int_\Omega (u^{\delta,\eps}(\tau,x)-u_0(x))\phi dx=\int_\Omega (u^{\eps}(\tau,x)-u_0(x))\phi dx,
	\end{equation}
for all $\phi\in H^1(\Omega)$. 
	
Next we show that $u^\eps$ is a weak solution of Problem \eqref{Pe0}. Since $u^{\delta,\eps}$ is a classical solution of Problem \eqref{Ped}, it is also a weak solution of this problem. Thus it satisfies
\begin{eqnarray}\label{weak1} 
&&\int_0^T \int_{ \Omega} (-u^{\delta,\eps} \phi_t + D_{\delta,\eps}(x) \nabla u^{\delta,\eps} \cdot \nabla \phi - f(u^{\delta,\eps}) \phi) dx dt\nonumber \\
&&\qquad= \int_{ \Omega} u_0(x) \phi dx -\int_\Omega u^{\delta,\eps}(T,x)\phi(T,x) dx,
\end{eqnarray}
for all test functions $\phi\in C^1([0,T]\times \overline{\Omega})$.

Since $u^{\delta,\eps}\to u^\eps$ in $L^2((0,T)\times \Omega)$ as $\delta\to0$, we deduce that
\[
\int_0^T \int_\Omega u^{\delta,\eps}\phi_t \to \int_0^T\int_\Omega u^\eps \phi_t \quad\text{as }\delta\to0.
\] 
Moreover, we remark that
\[
D_{\delta,\eps}\to D_\eps\quad\text{strongly in }L^2(\Omega)\quad \text{as }\delta\to0
\]
so that $D_{\delta,\eps}\nabla\phi \to D_\eps\nabla\phi$ strongly in $L^2((0,T)\times\Omega)$. This combined with the fact that as $\delta\to 0$,
\[\nabla u^{\delta,\eps}\rightharpoonup\nabla u^\eps\quad \text{weakly in } L^2((0,T)\times \Omega)\]
implies that
\[\int_0^T \int_\Omega D_{\delta,\eps}\nabla u^{\delta,\eps}\cdot\nabla\phi\to \int_0^T\int_\Omega D_\eps \nabla u^\eps\cdot\nabla \phi\]
as $\delta\to 0$. Since $u^{\delta,\eps}$ converges to $u^\eps$ a.e. in $Q_T$ as $\delta\to0$, it follows from the continuity of $f$ that $f(u^{\delta,\eps})$ converges to $f(u^\eps)$ a.e. in $Q_T$ as $\delta\to0$. Moreover, since that $|f(u^{\delta,\eps})|\le M_f$, we deduce from the dominated convergence theorem that
\[\int_0^T\int_\Omega f(u^{\delta,\eps})\phi\to\int_0^T\int_\Omega f(u^\eps)\phi\]
as $\delta\to 0$. 

Since $u^{\delta,\eps}(T,\cdot)\rightharpoonup u^\eps(T,\cdot)\mbox{ weakly in }(H^1)^*(\Omega)$ by \eqref{3.7},
\[\int_\Omega u^{\delta,\eps}(T,x)\phi(T,x) dx\to \int_\Omega u^{\eps}(T,x)\phi(T,x) dx\quad\mbox{ as }\delta\to0.\]
Therefore the passing to the limit $\delta\to0$ in \eqref{weak1}, we conclude that $u^\eps$ satisfies the integral equality for all $\phi\in  C^1([0,T]\times \overline{\Omega})$
\begin{eqnarray}\label{integral eq}
&&\int_0^T \int_{ \Omega} (-u^{\eps} \phi_t + D_\eps(x) \nabla u^{\eps} \cdot \nabla \phi - f(u^{\eps}) \phi) dx dt\nonumber\\ 
&&\qquad= \int_{ \Omega} u_0(x) \phi dx - \int_\Omega u^\eps(T,x)\phi(T,x)dx.
\end{eqnarray}
Since $ C^1([0,T]\times \overline{\Omega})$ is dense in $H^1((0,T)\times\Omega)$, we deduce that the identity \eqref{integral eq} also holds for all $\phi\in H^1((0,T)\times\Omega)$.

Next we prove the uniqueness of the weak solution. First we remark
\[u^\eps_t = \nabla\cdot(D_\eps(x)\nabla u^\eps) + f(u^\eps)\quad\text{in }L^2(0,T;(H^1)^*(\Omega)).\]
Suppose that there exist two solutions $u^\eps_1$ and $u^\eps_2$ of Problem \eqref{Ped}. Setting $w:=u^\eps_2-u^\eps_1$, we have that
\[w_t = \nabla\cdot(D_\eps(x)\nabla w)+f(u^\eps_2)-f(u^\eps_1)\quad\text{in }L^2(0,T;(H^1)^*(\Omega)),\]
which implies
\[\int_0^\tau \langle w_t,w\rangle dt = \int_0^\tau \langle \nabla\cdot (D_\eps(x)\nabla w),w\rangle + \langle f(u^\eps_2)-f(u^\eps_1),w\rangle dt,\]
for all $0\le\tau\le T$, that is
\begin{align*}
\frac{1}{2}\int_\Omega w^2(\tau) &+ \int_0^\tau \int_{\Omega} D_\eps(x)|\nabla w|^2 dxdt = \int_0^\tau\int_\Omega w\left(\int_{u^\eps_1}^{u^\eps_2}f'(s) ds \right) dx dt\\
&= \int_0^\tau \int_\Omega w (u^\eps_2-u^\eps_1)\left(\int_0^1 f'(\theta u^\eps_1 + (1-\theta) u^\eps_2) d\theta \right) dx dt\\
&\le  \tilde{M}_f \int_0^\tau\int_\Omega w^2 dx dt,
\end{align*}
where $\tilde{M}_f :=\sup_{s\in[0,M]} f'(s)$, for all $0\le \tau \le T$. Applying Gronwall's lemma, we deduce that $w=0$ a.e. in $Q_T$.
\end{proof}

\section{Singular limit $\eps \to 0$} \label{sec:convergence}

In this section, we prove that the solution of Problem \eqref{Pe0} converges strongly in $L^2((0,T)\times \Omega_1)$ to the solution of \eqref{P00}
\begin{equation}\label{P00}
\begin{cases}
u_t= \Delta u+f(u), &(t,x)\in(0,T)\times\Omega_1\\
\displaystyle{\frac{\partial u}{\partial \nu}} = 0,  &(t,x)\in (0,T)\times \partial\Omega_1 \\
u(0,x) = u_0(x), & x\in\Omega_1,
\tag{$P_0^0$}
\end{cases} 
\end{equation}
and that the solution of Problem \eqref{Pe0} converges weakly in $L^2((0,T)\times \Omega_0)$ to the solution of Problem \eqref{Q00}
\begin{equation}\label{Q00}
	\begin{cases}u_t = f(u), & (t,x)\in(0,T)\times \Omega_0\\
		u(0,x) = u_0(x),&x\in\Omega_0.
	\end{cases}\tag{$Q_0^{0}$}
\end{equation}
The solution of \eqref{P00} is defined in a weak sense.
\begin{definition}
A function $u\in C([0,T];L^2(0,T))\cap L^2((0,T);H^1( \Omega_1)) \cap L^\infty((0,T) \times  \Omega_1)$ is called a weak solution of \eqref{P00} if
\begin{eqnarray}\label{4.1}
&&\int_0^T \int_{ \Omega_1} (-u \phi_t +  \nabla u \cdot \nabla \phi - f(u) \phi) dx dt \nonumber\\
&&\qquad= \int_{ \Omega_1} u_0(x) \phi(0,x) dx - \int_{ \Omega_1} u(T,x) \phi(T,x) dx,
\end{eqnarray}
for any test function $\phi\in H^1((0,T)\times \Omega_1)$.
\end{definition}
It is a standard that Problem \eqref{P00} possesses a weak solution and the weak solution is unique. Furthermore, this weak solution is actually a classical solution $u\in C^{1,2}((0,T)\times \Omega_1)\cap C([0,T]\times\overline{\Omega}_1)$. Therefore, to complete the proof of Theorem \ref{thm:convergence1}, we need to show that the limit of $u^\eps$ converges strongly in  $L^2((0,T)\times \Omega_1)$ and satisfies \eqref{4.1}.

\begin{proof}[\bf Proof of Theorem 1.2]
Since $0\le u^\eps\le M$ and \eqref{2.3}, there exists a function $\bar{u}\in L^\infty((0,T)\times \Omega)$ and a subsequence $\{u^{\eps_n}\}$ such that
\[u^{\eps_n} \rightharpoonup \bar{u} \quad \mbox{weakly in $L^2((0,T)\times \Omega)$ as $\eps_n\to0$,}\]
\[u^{\eps_n}_t \rightharpoonup \bar{u}_t \quad \mbox{weakly in $L^2(0,T;(H^1)^*(\Omega))$ as $\eps_n\to0$.}\]
%

Next we show that $u^{\eps_n} \to \bar{u}$ strongly in $L^2((0,T)\times\Omega_1)$. Indeed, the sufficient estimates on differences of time and space translates follow from \eqref{3.5} and \eqref{time trans omega1}. Moreover, since $u^\eps$ is bounded in $L^\infty((0,T)\times \Omega)$,
the property (2) in the Fr\'echet-Kolmogorov theorem is satisfied. 
Thus, we can apply Fr\'echet-Kolmogorov theorem. We conclude that there exists a subsequence of $\{u^{\eps_n}\}$ which we denote again by $\{u^{\eps_n}\}$ such that 
\[u^{\eps_n} \to \bar{u} \quad \mbox{strongly in $L^2((0,T)\times \Omega_1)$ and a.e. in $(0,T)\times \Omega_1$ as $\eps_n\to0$}.\]
 We will show below that $\bar{u}$ coincides with the unique solution of Problem \eqref{P00}. 
 \begin{lemma} \label{laplacian omega2}
For each $\phi\in H^1((0,T)\times\Omega)$,
\[
\int_0^T\int_{\Omega_0} D^\eps(x) \nabla u^\eps \cdot \nabla \phi dxdt \to 0, \text{ as }\eps \to 0.
\]
\end{lemma}
\begin{proof}
We deduce from \eqref{2.2} that there exists a positive constant $C$ such that 
\[\|\nabla u^\eps\|_{L^2((0,T)\times \Omega)}\le \displaystyle{\frac{C}{\sqrt{\eps}}}\]
so that by the Cauchy-Schwarz inequality, 
\begin{align*}
|\int_0^T\int_{\Omega_0} D_\eps(x) \nabla u^\eps \cdot \nabla \phi dxdt| & \le \eps \|\nabla u^\eps\|_{L^2((0,T)\times \Omega_0)} \|\nabla \phi\|_{L^2((0,T)\times  \Omega_0)}\\
& \le \tilde{C}\sqrt{\eps}, \text{ as }\eps \to 0.
\end{align*}
\end{proof}
It follows from \eqref{2.2} that
\begin{equation*}
\nabla u^{\eps} \rightharpoonup \nabla \bar{u}  \quad \mbox{ weakly in } L^2((0,T)\times \Omega_1),
\end{equation*}
and since $|f(u^{\eps_n})|\le M_f$, we deduce that there exists a function
\[\chi\in L^\infty((0,T)\times\Omega)\]
and a subsequence of $\{u^{\eps_n}\}$ which we denote again by $\{u^{\eps_n}\}$ such that
\[f(u^{\eps_n})\rightharpoonup \chi \quad \mbox{weakly in $L^2((0,T)\times \Omega)$ as $\eps_n\to 0$.}\]
Moreover, since $u^{\eps_n}\to \bar{u}$ a.e. in $(0,T)\times \Omega_1$, 
\[f(u^{\eps_n})\to f(\bar{u}) \quad \mbox{a.e. in $(0,T)\times \Omega_1$}\]
and by Lebesgue's dominated convergence theorem,
\[f(u^{\eps_n})\to f(\bar{u}) \quad \mbox{strongly in $L^1((0,T)\times \Omega_1)$ as $\eps_n\to 0$.}\]
Thus 
\begin{equation}\label{chi:Omega_1}
	\chi = f(\bar{u})\quad \mbox{a.e in $(0,T)\times \Omega_1.$}
\end{equation}
We rewrite \eqref{integral eq} in the form
\begin{align*}
&\int_0^T \int_{\Omega_1}\{-u^{\eps_n}\phi_t + \nabla u^{\eps_n}\cdot\nabla\phi - f(u^{\eps_n})\phi\}\ dxdt\\
&\qquad + \int_0^T \int_{\Omega_0}\{-u^{\eps_n}\phi_t +  \eps\nabla  u^{\eps_n}\cdot\nabla\phi - f(u^{\eps_n})\phi\}\ dxdt\\
=&\int_{\Omega_1} u_0(x)\phi(0,x) dx +\int_{\Omega_0} u_0(x)\phi(0,x) dx - \int_{\Omega} u^{\eps_n}(T,x)\phi(T,x) dx
\end{align*}
for all $\phi\in H^1((0,T)\times \Omega)$, in which we let $\eps_n\to 0$ to obtain 
\begin{eqnarray}\label{weakform}
&\ds\int_0^T \int_{\Omega_1}(-\bar{u}\phi_t + \nabla\bar{u}\cdot\nabla\phi - f(\bar{u})\phi) dx dt + \int_0^T \int_{\Omega_0}(-\bar{u} \phi_t - \chi\phi) dx dt\nonumber\\
&\ds\quad=\int_{\Omega_1} u_0(x)\phi(0,x) dx + \int_{\Omega_0} u_0(x)\phi(0,x) dx - \int_{\Omega} \bar{u}(T,x)\phi(T,x) dx,
\end{eqnarray}
for all $\phi\in H^1((0,T)\times\Omega)$, since $u^{\eps_n}(T,x)\rightharpoonup\bar{u}(T,x)$ weakly in $(H^1)^*(\Omega)$. Take $\phi $ arbitrary in $H^1((0,T) \times \Omega)$ such that $\phi(T,x)=0$ for $x\in\Omega_0$ and $\phi=0$ in $(0,T)\times\Omega_1$. Then, \eqref{weakform} becomes 

\[
\int_0^T \int_{\Omega_0} (-\bar{u} \phi_t  - \chi \phi) dx dt  = \int_{\Omega_0 }u_0(x) \phi(0,x) dx,
\]
that is 
\begin{equation}\label{4.3}
- \int_0^T \langle {\bar{u}}_t,\phi\rangle dt - \int_0^T \int_{\Omega_0} \chi \phi dx dt  + \int_{\Omega_0 }{(\bar{u}}(0,x) -  u_0(x)) \phi(0,x) dx = 0.
\end{equation}
Taking $\phi$ arbitrary in $C_c^\infty((0,T) \times \Omega_0)$, we deduce that
\begin{equation}\label{eq:Omega2}
\bar{u}_t  =  \chi \quad\mbox{ for a.e. }(t,x)\in(0,T)\times\Omega_0.
\end{equation}
Now taking $\phi$ arbitrary in $H^1((0,T)\times\Omega)$ such that $\phi(T,x)=0$ for $x\in\Omega_0$ and $\phi=0$ in $(0,T)\times\Omega_1$, we deduce that $u(0,x)=u_0(x)$ a.e. for $x\in\Omega_0$.
\begin{lemma}
There holds $\chi = f(\bar{u})$ in $L^\infty((0,T)\times \Omega)$. 
\end{lemma}
	\begin{proof}
		We already proved that this property holds in $(0,T)\times \Omega_1$, but not yet in the whole domain $(0,T)\times\Omega$. We define 
		\[c:= \sup_{s\in[0,M]} f'(s),\]
		and set 
		\[v^\eps := e^{-ct}u^\eps,\quad \bar{v} := e^{-ct}\bar{u}\]
		and 
		\[g(t,v):= e^{-ct}f(e^{ct}v)-cv,\quad \tilde{\chi} := e^{-ct}(\chi-c\bar{u}),\quad \psi:=e^{ct}\phi.\] 
		We will apply a classical monotonicity argument which can be found for instance in \cite{Marion}. We recall that by Definition \ref{def:weak}, we have 
		\begin{eqnarray*}
					&&\int_0^T \int_\Omega (-u^\eps\phi_t + D_\eps(x)\nabla u^\eps\cdot\nabla\phi -f(u^\eps)\phi) dxdt \\
					&&\qquad= \int_\Omega u_0(x)\phi(0,x) dx - \int_\Omega u^\eps(T,x)\phi(T,x) dx,
					\end{eqnarray*}
				for all $\phi\in H^1((0,T)\times \Omega)$, which implies that
				\begin{eqnarray*}
					&&\int_0^T \int_\Omega (-v^\eps\psi_t + D_\eps(x)\nabla v^\eps\cdot\nabla\psi - g(t,v^\eps)\psi) dxdt \\
					&&\qquad= \int_\Omega u_0(x)\psi(0,x) dx - \int_\Omega v^\eps(T,x)\psi(T,x) dx,
				\end{eqnarray*}
				  for all $\psi\in H^1((0,T)\times \Omega)$, so that also
				\begin{equation}\label{4.5}
					\int_0^T \langle v^\eps_t,\psi\rangle  dt + \int_0^T\int_\Omega (D_\eps(x)\nabla v^\eps\cdot \nabla \psi - g(t,v^\eps)\psi)dxdt =0,
				\end{equation}
				 for all $\psi\in L^2(0,T; H^1(\Omega))$. We have that
		\begin{equation*}
			g_v(t,v^\eps) = e^{-ct}e^{ct}f'(e^{ct}v^\eps)-c= f'(e^{ct}v^\eps)-c\le0,
		\end{equation*}
	where we have used the fact that
	\begin{equation*}
		0\le u^\eps=e^{ct}v^\eps\le M.
	\end{equation*}
Since $g$ is decreasing in $v$ we deduce that
		\begin{equation*}
			\int_0^T \int_\Omega (g(t,v^\eps)-g(t,\psi) )(v^\eps-\psi) dx dt \le 0,
		\end{equation*}
		for $\psi\in L^2(0,T;H^1(\Omega))$. Thus,
		\begin{align}
			0\ge& \liminf_{\eps\to0}\int_0^T \int_\Omega (g(t,v^\eps)-g(t,\psi) )(v^\eps-\psi) dx dt \label{4.6}\\
			&= \liminf_{\eps\to0} (\int_0^T \int_\Omega g(t,v^\eps) v^\eps dx dt\nonumber\\
			&-\int_0^T \int_\Omega g(t,\psi) v^\eps dx dt -\int_0^T \int_\Omega (g(t,v^\eps)-g(t,\psi))\psi dx dt)	\nonumber
		\end{align}
	Substituting $\psi = v^\eps$ in \eqref{4.5}, we obtain
	\begin{equation}\label{4.7}
		\int_0^T\langle v^\eps_t,v^\eps\rangle  dt + \int_0^T\int_\Omega (D_\eps(x)\nabla v^\eps\cdot \nabla v^\eps - g(t,v^\eps)v^\eps)dxdt =0.
	\end{equation}
Combining \eqref{4.6} and \eqref{4.7} yields
	\begin{align*}	 
		0\ge& \liminf_{\eps\to0} (\int_0^T \langle v^\eps_t, v^\eps\rangle dt + \int_0^T \int_\Omega D_\eps(x)|\nabla v^\eps |^2dx dt\\
			&-\int_0^T \int_\Omega g(t,\psi) v^\eps dx dt -\int_0^T \int_\Omega (g(t,v^\eps)-g(t,\psi))\psi dx dt)\\
			=&\liminf_{\eps\to0} (\frac{1}{2}\int_\Omega |v^\eps(T)|^2 dx-\frac{1}{2}\int_\Omega |u_0|^2 dx + \int_0^T \int_\Omega D_\eps(x)|\nabla v^\eps|^2 dx dt\\
			&-\int_0^T \int_\Omega g(t,\psi) v^\eps dx dt -\int_0^T \int_\Omega (g(t,v^\eps)-g(t,\psi))\psi dx dt).
				\end{align*}
			Thus letting $\eps$ tends to $0$, we deduce that
			\begin{align}\label{ineq:vbar}
			0\ge & \frac{1}{2}\int_\Omega |\bar{v}(T)|^2 dx - \frac{1}{2} \int_\Omega |u_0|^2 dx + \int_0^T \int_{\Omega_1} |\nabla \bar{v}|^2  dx dt\\
			&-\int_0^T \int_\Omega g(t,\psi) \bar{v} dx dt -\int_0^T \int_\Omega (\tilde{\chi}-g(t,\psi))\psi dx dt\nonumber
			\end{align}
		
We remark that \eqref{chi:Omega_1} and \eqref{weakform} imply
	
	\begin{equation*}
		\int_0^T \langle \bar{u}_t,\phi\rangle  dt +\int_0^T \int_{\Omega_1} (\nabla\bar{u}\cdot\nabla\phi - \chi\phi) dx dt - \int_0^T\int_{\Omega_0}\chi\phi dx dt = 0,
	\end{equation*}
	for all $\phi \in L^2(0,T;H^1(\Omega))$, which in turn implies that
	\begin{equation*}
		\int_0^T \langle \bar{v}_t+c\bar{v},\psi\rangle  dt +\int_0^T \int_{\Omega_1} (\nabla\bar{v}\cdot\nabla\psi) dx dt - \int_0^T\int_{\Omega}e^{-ct}\chi\psi dx dt = 0,
	\end{equation*}
	\begin{equation}\label{eq:weakform vbar}
		\int_0^T \langle \bar{v}_t,\psi\rangle  dt +\int_0^T \int_{\Omega_1} (\nabla\bar{v}\cdot\nabla\psi) dx dt - \int_0^T\int_{\Omega}\tilde{\chi}\psi dx dt = 0,
	\end{equation}
for all $\psi \in L^2(0,T;H^1(\Omega))$. Setting $\psi=\bar{v}$ in \eqref{eq:weakform vbar} we deduce that
	\begin{equation}\label{eq:vbar}
 \frac{1}{2}\int_\Omega |\bar{v}(T)|^2 dx - \frac{1}{2} \int_\Omega |u_0|^2 dx + \int_0^T \int_{\Omega_1} |\nabla\bar{v}|^2 dx dt - \int_0^T\int_{\Omega}\tilde{\chi}\bar{v} dx dt = 0.
\end{equation}
Substituting \eqref{eq:vbar} into \eqref{ineq:vbar} yields
\begin{equation*}
	\int_0^T \int_{\Omega} (\tilde{\chi}-g(t,\psi))(\bar{v}-\psi) dx dt\le 0,
\end{equation*}
for all $\psi\in L^2(0,T;H^1(\Omega))$.

Taking $\psi = \bar{v}-\lambda w$, $\lambda>0$, $w\in L^2(0,T;H^1(\Omega))$, we deduce that
\begin{equation}\label{4.12}
	\int_0^T \int_\Omega (\tilde{\chi}-g(t,\bar{v}-\lambda w))w\le 0,
	\end{equation}
	for all $w\in L^2(0,T;H^1(\Omega))$. Letting $\lambda\to0$ in \eqref{4.12}, we obtain
	\begin{equation}\label{4.13}
		\int_0^T \int_\Omega (\tilde{\chi}-g(t,\bar{v}))w\le0,
\end{equation}
	for all $w\in L^2(0,T;H^1(\Omega))$. Setting $w=-w$ in \eqref{4.13}, we deduce that
	\begin{equation*}
		\int_0^T \int_\Omega (\tilde{\chi}-g(t,\bar{v}))w=0,
	\end{equation*}
	 which yields that
		\[\tilde{\chi} = g(t,\bar{v})\quad \mbox{a.e,}\]
		or else
		\begin{equation}\label{weakconvergence}
			\chi = f(\bar{u})\quad \mbox{a.e.}
			\end{equation}
	\end{proof}
	
Taking $\phi\in C^1_c([0,T)\times\Omega_1)$ arbitrary in \eqref{weakform} yields
\begin{equation}\label{weakform1}
	\int_0^T \int_{\Omega_1}(-\bar{u}\phi_t + \nabla\bar{u}\cdot\nabla\phi - f(\bar{u})\phi) dx dt = \int_{\Omega_1} u_0(x)\phi(0,x) dx ,
\end{equation}
for all $\phi\in C^1_c([0,T)\times\Omega_1)$. Then taking $\phi $ arbitrary in $C^\infty_c((0,T)\times \Omega_1)$, 
we deduce that $\bar{u}$ satisfies the partial differential equation 
\[
{\bar{u}}_t= \Delta {\bar{u}}+f({\bar{u}}),\quad \mbox{ in the sense of distributions in } (0,T)\times \Omega_1.
\]
Next we take $\phi $ arbitrary in $C^\infty_c([0,T)\times \Omega_1)$, to deduce that
\[
{\bar{u}}(0,x) = u_0(x), \quad x\in\Omega_1.
\]
Finally, taking $\phi $ arbitrary in $C^\infty_c((0,T)\times \overline{\Omega}_1)$ in \eqref{weakform1}, we deduce that 
\[
\displaystyle{\frac{\partial \bar{u}}{\partial \nu}} = 0, \quad \mbox{ in the sense of distributions on } (0,T)\times \Gamma.
\]
It then follows from standard arguments that $\bar{u}$ coincides with the unique classical solution $u$ of Problem \eqref{P00}.
\end{proof}

\begin{proof}[\bf Proof of Theorem \ref{thm:convergence2}]
	The result of Theorem \ref{thm:convergence2} follows from \eqref{eq:Omega2} and \eqref{weakconvergence}.
\end{proof}

\section{Numerical simulation}

In this section, we observe the evolution of the solution numerically and test the appearance of the Neumann boundary condition along the interior boundary $\partial\Omega_1$ when $\eps\to0$. For the test, we consider one dimensional problem with $\Omega=(0,4)$ and its interior domain $\Omega_1=(1,3)$. Then, the equation \eqref{Pe0} is written as
\begin{equation}\label{6.1}
	\begin{cases}
		u_t = (D_\eps u_x )_x + f(u), &t>0,\ 0<x<4, \\
		u_x(t,0)=u_x(t,4) =0,&t>0,\\
		u(0,x) = u_0(x),&0<x<4.
	\end{cases}
\end{equation}
We take the initial value, reaction function, and diffusivity as
\begin{equation}\label{6.2}
	u_0(x)=\sin(\pi x/4),\quad f(u)=-u(u-1/3)(u-1),
\end{equation}
and
\begin{equation}\label{6.3}
	D_\eps(x)=\begin{cases}
		1, & 1\le x\le3,\\
		\eps,& \text{otherwise}.
	\end{cases}
\end{equation}
The snapshots of the numerical solutions of \eqref{6.1}--\eqref{6.3} are given in Figure \ref{fig2} at the moment $t=0.1$ with four cases  of different epsilons. For this computation, the MTLAB function `pdepe' is used with a mesh size $\triangle x=0.001$. The behavior of the solution at the interface $x=1$ is magnified. We can observe a development of discontinuity of the gradient $|\nabla u|$ and the solution itself $u$. This is due to the continuity relation of the flux at the interface, which is $\eps |\nabla u(1-)|=|\nabla u(1+)|$.

\begin{figure}[h]
	\centering
	\includegraphics[width=0.49\textwidth]{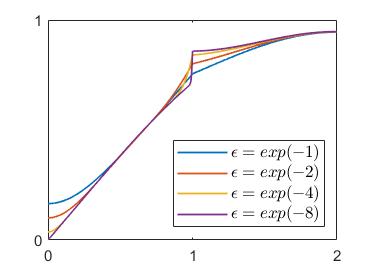}
	\includegraphics[width=0.49\textwidth]{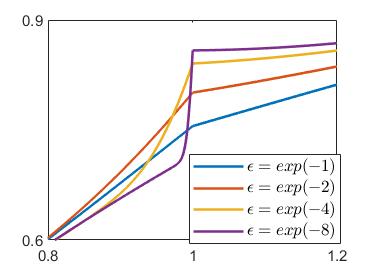}
	\caption{\small Solution snapshots at $t=0.1$ for different $\eps$ in the half of the domain and in a magnified one at the interface.}\label{fig2}
\end{figure}

The limit satisfies Problem \eqref{P00} in the interior domain which is
\begin{equation}\label{6.4}
	\begin{cases}u_t = u_{xx} + f(u), & t>0,\ 1<x<3,\\
		u_x(t,1)=u_x(t,3) = 0,& t>0,\\
		u(0,x) = u_0(x),&1<x<3.
	\end{cases}
\end{equation}
The same initial value in \eqref{6.2} is taken from the interior domain $(1,3)$. The snapshot of the numerical solution of the Neumann problem \eqref{6.4} is given in Figure \ref{fig3} together with the snapshots in Figure \ref{fig2} in the whole interior domain $\Omega_1=(1,3)$. We can observe that the solutions of \eqref{6.1} converge to the solution of the Neumann problem \eqref{6.4} as $\eps\to0$ from the below.

\begin{figure}[h]
	\centering
	\includegraphics[width=0.55\textwidth,height=47mm]{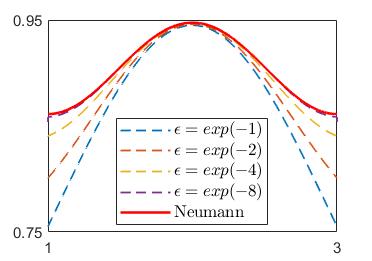}
	\caption{\small Convergence to Neumann problem \eqref{P00} ($t=0.1$).}\label{fig3}
\end{figure}

In Figure \ref{fig4}, the gradient values $|\nabla u(t,1+)|$ at the interface are given reducing $\eps$ with $\eps=e^{-j}$ for $j=0,\cdots,8$. They are given at three different time moments, $t=0.01,0.03$ and $0.05$. We can see that these gradient from the inside domain decays to zero as $\eps\to0$. The second figure is displayed with log-log scale. We can see that $\ln(|\nabla u|)$ and $\ln(\eps)$ satisfy a linear relation,
\[
\ln|\nabla u|=a\ln(\eps)+b,
\]
which is equivalently written as
\[
|\nabla u|=e^b\eps^a.
\]
Then, the three cases with $t=0.01,0.04$ and $0.09$ gives convergence order of
\[
|\nabla u(t,1+)|\cong
\begin{cases}
	e^{-0.5941}\eps^{0.4951},&t=0.01\\
	e^{-0.6112}\eps^{0.5172},&t=0.04\\
	e^{-0.6650}\eps^{0.5333},&t=0.09.
\end{cases}
\]
We can see that the gradient of the solution at the interface converges to zero as $\eps\to0$ approximately with order $O(\sqrt{\eps})$ and the convergence order increases as the time variable $t$ increases. This order is not surprising. For a fixed time $t>0$, the gradient $|\nabla u|$ in the region with diffusivity $D=1$ is supposed to be bounded. In the region with $D=\eps$, the gradient $|\nabla u|$ is of order $(1/\sqrt{\eps})$, which is from parabolic rescaling of the space variable. Therefore, the interface condition
\[
|\nabla u(t,1+)|=\eps|\nabla u(t,1-)|=O(\sqrt{\eps})\quad\mbox{as}\quad\eps\to0.
\]

\begin{figure}[h]
	\centering
	\includegraphics[width=0.49\textwidth]{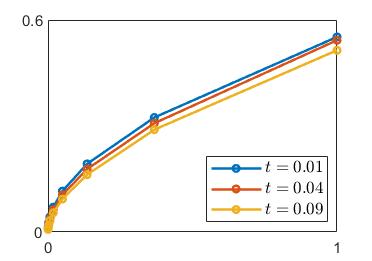}
	\includegraphics[width=0.49\textwidth]{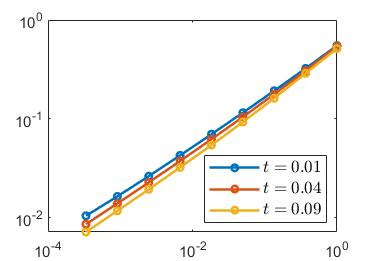}
	\caption{\small The gradient $|\nabla u|$ decays to zero as $\eps\to0$.}\label{fig4}
\end{figure}

In Figure \ref{fig5}, five snapshots of numerical solutions are given which show the long time behavior of the solution when $\eps$ is small. We took $\eps=e^{-16}$ for this simulation. We can see that a discontinuity is developing at the interface $x=1$ very quickly ($t=0.1$). However, the long time dynamics takes it over and eventually converges to a step function with the discontinuity at $x=0.4327$ decided by the initial value and the reaction function. This behavior is expected since the solution behavior in the region with $D=\eps$ follows the ODE solution $\dot u=f(u)$ as $\eps\to0$. Since the unstable steady state is $u=1/3$ and the initial value takes the value at $x=0.4327$, i.e., $u_0(0.4327)\cong1/3$, the discontinuity develops at the point $x=0.4327$.

\begin{figure}[h]
	\centering
	\includegraphics[width=0.65\textwidth]{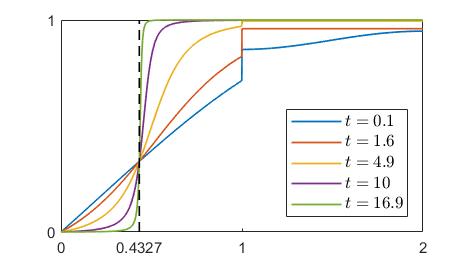}
	\caption{\small Solution snapshots with $\eps=e^{-16}$. Initially, a discontinuity develops at the interface $x=1$. Then, asymptotically, the solution converges to piecewise constant profile with a discontinuity at $x=0.4327$.}\label{fig5}
\end{figure}


\end{document}